\documentclass[10pt,aps,notitlepage,onecolumn,pra,superscriptaddress]{revtex4-1}

\usepackage{algorithm,algpseudocode,amsfonts,amsmath,amssymb,amsthm,colortbl,dsfont,graphicx,hyperref,hyphenat,lmodern,mathtools,microtype,multirow,xcolor}
\usepackage[UKenglish]{babel}
\usepackage[capitalize,nameinlink]{cleveref}
\usepackage[scaled=0.8]{FiraMono}
\usepackage[T1]{fontenc}
\usepackage[utf8]{inputenc}
\usepackage[UKenglish]{isodate}
\usepackage[caption=false]{subfig}

\definecolor{darkred}{rgb}{0.5,0,0}
\definecolor{darkgreen}{rgb}{0,0.5,0}
\definecolor{darkblue}{rgb}{0,0,0.5}
\hypersetup{colorlinks,breaklinks,linkcolor=darkred,citecolor=darkgreen,urlcolor=darkblue}

\DeclareMathOperator*{\argmax}{arg\,max}
\DeclareMathOperator{\codim}{codim}
\DeclareMathOperator{\conv}{conv}

\newtheorem{problem}{Problem}
\newtheorem{proposition}{Proposition}

\renewcommand{\leq}{\leqslant}
\renewcommand{\geq}{\geqslant}

\newcommand\hb[2]{\genfrac{}{}{0pt}{}{#1}{#2}}
\newcommand{\bd}{\mathrm{d}}
\newcommand{\be}{\mathrm{e}}
\newcommand{\cc}{\cellcolor[gray]{0.9}}
\newcommand\cut[2]{\mathrm{SDP}_{#1}^{#2}}
\newcommand{\id}{\mathds{1}}
\newcommand{\innp}[2]{\left\langle{#1},{#2}\right\rangle}

\newcommand\KG[2]{K_G\left({#1}\ifx#2\empty\else\mapsto\fi{#2}\right)}
\newcommand{\NN}{\mathbb{N}}

\newcommand{\RR}{\mathbb{R}}
\newcommand\sdp[2]{\mathrm{SDP}_{#1}\ifx#2\empty\else\left({#2}\right)\fi}
\newcommand{\ta}{\tilde{a}}
\newcommand{\tb}{\tilde{b}}
\newcommand{\tinnp}[2]{\langle{#1},{#2}\rangle}

\makeatletter
\def\CT@@do@color{%
  \global\let\CT@do@color\relax
  \@tempdima\wd\z@
  \advance\@tempdima\@tempdimb
  \advance\@tempdima\@tempdimc
  \advance\@tempdimb0.9\tabcolsep
  \advance\@tempdimc\tabcolsep
  \advance\@tempdima2\tabcolsep
  \kern-\@tempdimb
  \leaders\vrule
  \hskip\@tempdima\@plus 1fill
  \kern-\@tempdimc
\hskip-\wd\z@ \@plus -1fill }
\makeatother

\begin{document}

\author{Sébastien Designolle}
\affiliation{Zuse Institute Berlin, 14195 Berlin, Germany}
\affiliation{Inria, ENS de Lyon, UCBL, LIP, 69342, Lyon Cedex 07, France}
\author{Tamás Vértesi}
\affiliation{HUN-REN Institute for Nuclear Research, 4001 Debrecen, Hungary}
\author{Sebastian Pokutta}
\affiliation{Zuse Institute Berlin, 14195 Berlin, Germany}
\affiliation{Institut für Mathematik, Technische Universität Berlin, Straße des 17. Juni 136, 10623 Berlin, Germany}

\title{Better bounds on finite-order Grothendieck constants}
\date{5th October 2026}

\begin{abstract}
  Grothendieck constants $\KG{d}{}$ bound the advantage of $d$-dimensional strategies over $1$-dimensional ones in a specific optimisation task.
  They have applications ranging from approximation algorithms to quantum nonlocality.
  However, apart from $d=2$, their values are unknown.
  Here, we exploit a recent Frank-Wolfe approach to provide good candidates for lower bounding some of these constants.
  The complete proof relies on solving difficult binary quadratic optimisation problems.
  For $d\in\{3,4,5\}$, we construct specific rectangular instances that we can solve to certify better bounds than those previously known; by monotonicity, our lower bounds improve on the state of the art for $3\leqslant d\leqslant9$.
  For $d\in\{4,7,8\}$, we exploit elegant structures to build highly symmetric instances achieving even greater bounds; however, we can only solve them heuristically.
  We also recall the standard relation with violations of Bell inequalities and elaborate on it to interpret generalised Grothendieck constants $\KG{d}{2}$ as the advantage of complex $d$-dimensional quantum mechanics over real qubit quantum mechanics.
  Motivated by this connection, we also improve the bounds on $\KG{d}{2}$.
\end{abstract}

\maketitle

\section{Introduction}

Published in French in a Brazilian journal, Grothendieck's pioneering work on Banach spaces from 1953~\cite{Gro53}, now informally known as his \emph{Résumé}, long remained unnoticed.
In 1968, Lindenstrauss and Pe\l{}czy\'nski~\cite{LP68} discovered it and rephrased the main result, the Grothendieck inequality, which proves a relationship between three fundamental tensor norms through the so-called Grothendieck constant, denoted $K_G$.
Since then, this far-reaching theorem has found numerous applications~\cite{Pis11}, in particular in combinatorial optimisation where it is at the heart of an algorithm to approximate the cut-norm of a matrix~\cite{KN12}.

Quantum information is another field where this result is influential: following early observations by Tsirelson~\cite{Tsi87}, an explicit connection has been established with the noise robustness in Bell experiments~\cite{Bel64,AGT06}.
These experiments aim at exhibiting a fascinating property of quantum mechanics in correlation scenarios: nonlocality~\cite{BCP+14}.
The link with Grothendieck's theorem has led to a surge of interest in the value of $\KG{3}{}$, the Grothendieck constant of order three.
Many works have thus demonstrated increasingly precise lower bounds~\cite{CHSH69,Ver08,HLZ+15,BNV16,DBV17,DIB+23} and upper bounds~\cite{Kri79,HQV+17,DIB+23} on its value.
More recently, a numerical method has also been developed to come up with an even more precise (but not provable) estimate of $\KG{3}{}$~\cite{SM25}.

For Grothendieck constants of higher orders, the link with quantum nonlocality remains~\cite{AGT06,VP08}.
However, the bounds on their values have been less studied and are less tight~\cite{Kri79,FR94,Ver08,BBT11,HLZ+15,DBV17}.
There are quite a few difficulties that explain this relative scarcity of results, many of them being manifestations of the curse of dimensionality.
Finding suitable high-dimensional ansätze indeed becomes increasingly hard and the resulting instances involve sizes that rapidly become intractable.

In this article, we combine the recent projection technique from~\cite{DIB+23} with the powerful solver developed in~\cite{DBV17} to obtain better lower bounds on $\KG{d}{}$ for $3\leq d\leq9$.
Following~\cite{FR94}, we also consider symmetric structures in high dimensions emerging from highly symmetric line packings~\cite{FJM18} to suggest even better bounds on $\KG{4}{}$, $\KG{7}{}$, and $\KG{8}{}$, that we unfortunately cannot prove as they involve optimisation problems that we only solve heuristically.
We also consider the generalised Grothendieck constants $\KG{d}{2}$ and interpret them as the advantage of $d$-dimensional quantum mechanics over real qubit quantum mechanics, a fact that was already studied in~\cite{VP08,VP09}.
Our bounds on $\KG{d}{}$ are analytical, while those on $\KG{d}{2}$ strongly rely on numerical methods: we post-process the upper bound on $\KG{3}{2}$ to convert it into an exact result, but the lower bounds remain inaccurate.

We first formally define the constants we want to bound in \cref{sec:preliminaries} before presenting in \cref{sec:KGd_low} our method and main results on lower bounds on $\KG{d}{}$, summarised in \cref{tab:sym,tab:asy}.
Then we turn to generalised constants in \cref{sec:KGd2}, where we review existing bounds before deriving ours, which are particularly tight on the value of $\KG{3}{2}$.
Finally, we recall the connection with quantum mechanics in \cref{sec:quantum} and conclude in \cref{sec:discussion}.

\section{Preliminaries}
\label{sec:preliminaries}

Given a real matrix $M$ of size $m_1\times m_2$, we define
\begin{equation}
  \sdp{d}{M}=\max\left\{\sum_{x=1}^{m_1}\sum_{y=1}^{m_2} M_{xy}\innp{a_x}{b_y}~\Big|~\forall x\in[m_1],~a_x\in S^{d-1},~\forall y\in[m_2],~b_y\in S^{d-1}\right\},
  \label{eqn:SDPdM}
\end{equation}
where $[m]=\{1,\ldots,m\}$ and where $S^{d-1}$ is the unit sphere in $\RR^d$, which reduces to $S^0=\{-1,1\}$ for $d=1$.
By introducing the set
\begin{equation}
  \cut{d}{m_1,m_2}=\conv\left\{X~\big|~X_{xy}=\innp{a_x}{b_y},~\forall x\in[m_1],~a_x\in S^{d-1},~\forall y\in[m_2],~b_y\in S^{d-1}\right\},
  \label{eqn:cut}
\end{equation}
we get the equivalent definition
\begin{equation}
  \sdp{d}{M}=\max\left\{\innp{M}{X}~\big|~X\in\cut{d}{m_1,m_2}\right\}.
  \label{eqn:SDPdM_cut}
\end{equation}
The reason for this name comes from the interpretation of the Grothendieck constant that we are about to define in the context of rank-constrained semidefinite programming~\cite{BOV14}.
In the following, when $m_1=m_2$, we simply write $\cut{d}{m,m}=\cut{d}{m}$; also, when the size $m_1,m_2$ is either clear from the context or irrelevant, we will use the shorthand notation $\cut{d}{m_1,m_2}=\cut{d}{}$.

In essence, the Grothendieck inequality states that there exists a (finite) constant independent of the size of the matrix $M$ that bounds the ratio between the quantities $\sdp{d}{M}$ for various $d$.
More formally, given $n\leq d$, for all real matrices $M$, we have~\cite{Gro53}
\begin{equation}
  \sdp{d}{M}\leq\KG{d}{n}\sdp{n}{M},
  \label{eqn:Grothendieck}
\end{equation}
so that the exact definition of these \emph{generalised Grothendieck constants} is
\begin{equation}
  \KG{d}{n}=\sup\left\{\frac{\sdp{d}{M}}{\sdp{n}{M}}~\Big|~M\in\RR^{m_1\times m_2},\quad m_1,m_2\in\NN\right\}.
  \label{eqn:KGdn}
\end{equation}
The standard Grothendieck constant of order $d$ is obtained when $n=1$, in which case we use the shorthand notation $\KG{d}{1}=\KG{d}{}$.
The infinite-order Grothendieck constant $K_G$ initially studied in~\cite{Gro53} corresponds to the limit $\lim_{d\rightarrow\infty}\KG{d}{}$, but it is out of the scope of our work, although we briefly mention it in \cref{sec:discussion}.
We refer the reader interested in general aspects and further generalisations of these constants to~\cite{Bri11}.

\section{Lower bounds on \texorpdfstring{$\KG{d}{}$}{KG(d)}}
\label{sec:KGd_low}

Given the definition of $\KG{d}{}$ in \cref{eqn:KGdn}, any matrix $M$ automatically provides a valid lower bound.
However, there are two main difficulties when looking for \emph{good} lower bounds.
\begin{problem}
  Given $m_1$ and $m_2$, how to find a matrix $M$ such that the inequality in \cref{eqn:Grothendieck} is as tight as possible?
  \label{pb:tight}
\end{problem}
\begin{problem}
  Given such a matrix $M$, how to compute the resulting $\sdp{1}{M}$ or at least a close upper bound?
  \label{pb:qubo}
\end{problem}
\cref{pb:qubo} is the most limiting one, as the computation of $\sdp{1}{M}$ is equivalent to MaxCut and therefore NP-hard, so that the sizes of the matrices to consider are limited by the methods and resources available to compute this number.
In general, solving \cref{pb:tight} exactly is out of reach, but good candidates can be found and this suffices to derive bounds.

In the rest of this section, we first recall the method from~\cite{DIB+23} for addressing both problems above, in particular \cref{pb:tight}.
To illustrate the idea of the method, we give up on solving \cref{pb:qubo} for the sake of the elegance of the solution to \cref{pb:tight}, providing instances of remarkable symmetry for which future works may solve \cref{pb:qubo} to improve on the bounds on $\KG{4}{}$, $\KG{7}{}$, and $\KG{8}{}$.
We then focus on \cref{pb:qubo} and consider rectangular matrices allowing us to obtain certified bounds on $\KG{3}{}$, $\KG{4}{}$, and $\KG{5}{}$ that beat the literature up to $d=9$, by monotonicity.

\subsection{Obtaining facets of the symmetrised correlation polytope}
\label{sec:fw}

When $n=1$, the set $\cut{1}{}$ is a polytope, called the correlation polytope~\cite{Pit91}.
In~\cite{BNV16,DBV17,DIB+23} the method used to solve \cref{pb:tight} is to start from a point $P\in\cut{d}{}$ and to derive $M$ as a hyperplane separating $P$ from $\cut{1}{}$.
This hyperplane is obtained by solving the projection of $P$ onto the correlation polytope $\cut{1}{}$ via Frank-Wolfe algorithms.
We refer to~\cite{Pok24} for a gentle introduction to FW algorithms, to~\cite{BCC+25} for a complete review, and to~\cite[Appendix~C]{DIB+23} for the details of our implementation.
Note that some accelerations developed in~\cite{DIB+23,DVP24} and used in this work are now part of the \texttt{FrankWolfe.jl} package~\cite{BDH+24}.

In particular, the symmetrisation described in~\cite{DVP24} is crucial, but as it depends on the underlying group and its action on the matrices we consider, exposing the full structure of the \emph{symmetrised correlation polytope} for each group would be tedious.
Instead, we give a brief general definition of this polytope and refer to~\cite{DVP24} for a detailed example.
Given a group $G$ acting on $[m_1]$ and $[m_2]$ (by means of signed permutations), the symmetrised correlation polytope $G(\cut{1}{m_1,m_2})$ is the convex hull of the averages of all orbits of the vertices of $\cut{1}{m_1,m_2}$ under the action of $G$.
Note that this polytope lives in a subspace of the space invariant under the action of $G$ on matrices of size $m_1\times m_2$.

In \cref{fig1} we illustrate the various geometrical cases that we encounter in this work.
Their understanding justifies the procedure that we describe in \cref{alg:facet} to derive facets of the symmetrised correlation polytope. 
Importantly, for symmetric instances, the dimension of the ambient space is strictly smaller than $m_1m_2$.
In the following, we consistently denote these facets by $A$, while $M$ will indicate separating hyperplanes without this extra property.
Putative facets will also be denoted by $A$, that is, in cases where we rely on heuristic methods to compute $\sdp{1}{A}$.

These heuristic methods play an important role throughout the FW algorithm as they quickly give a good direction to make primal progress.
Here we present them in a generalised framework that will prove useful when generalising our algorithm to $\KG{d}{2}$ in \cref{sec:KGd2}.
At a given step of the FW algorithm minimising the squared distance to the set $\cut{n}{}$, finding the best direction with respect to the current negative gradient $M$ is exactly the problem in \cref{eqn:SDPdM_cut}.
This subroutine is called the Linear Minimisation Oracle (LMO) and, in the course of the algorithm, we usually use an alternating minimisation to obtain heuristic solutions that are enough to make progress~\cite[Appendix~B.1]{DIB+23}.
This procedure is known as a \emph{seesaw} optimisation in the quantum community~\cite{WW01}.

More formally, for all $x\in[m_1]$, we pick a random $a^h_x\in S^{n-1}$ and we compute, for all $y\in[m_2]$,
\begin{equation}
  b^h_y=\argmax_{b_y\in S^{n-1}}\innp{b_y}{\sum_{x=1}^{m_1}M_{xy}a^h_x},\qquad\text{that is,}\qquad b^h_y = \frac{\sum_{\vphantom{y}x=1}^{m_1} M_{xy}a^h_x}{\left\|\sum_{\vphantom{y}x=1}^{m_1} M_{xy}a^h_x\right\|}.
  \label{eqn:bh}
\end{equation}
In the unlikely case where the denominator happens to be zero, we set $b^h_y$ to be a predetermined vector, for instance, $(1,0,\ldots,0)\in S^{n-1}$.
Similarly, we then use these $b^h_y$ to compute
\begin{equation}
  a^h_x=\argmax_{a_x\in S^{n-1}}\innp{a_x}{\sum_{y=1}^{m_2}M_{xy}b^h_y},\qquad\text{that is,}\qquad a^h_x = \frac{\sum_{y=1}^{m_2} M_{xy}b^h_y}{\left\|\sum_{y=1}^{m_2} M_{xy}b^h_y\right\|},
  .
  \label{eqn:ah}
\end{equation}
and we repeat \cref{eqn:bh,eqn:ah} until the objective value in \cref{eqn:SDPdM} stops increasing, up to numerical precision when $n>1$.
Depending on the initial choice of $a^h_x$, the value attained will vary.
Therefore we restart the procedure a large number of times: from a few hundred or thousand within the FW algorithm, to $10^5$ in \cref{tab:sym_n2} and $10^8$ in \cref{tab:sym}.

In the following, we clearly mention when the last part of our method --- the computation of $\sdp{1}{M}$ (or $\sdp{2}{M}$ in \cref{sec:KGd2}) where $M$ is the last gradient returned by our FW algorithm --- is done with the heuristic procedure just described, which comes with no theoretical guarantee, or with a more elaborate solver for which the exact value can be certified or rigorously upper bounded.

\begin{figure}[ht!]
  \centering
  \subfloat[][]{
    \includegraphics[page=1]{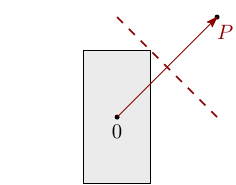}
    \label{fig1a}
  }
  \subfloat[][]{
    \includegraphics[page=2]{fig1.pdf}
    \label{fig1b}
  }
  \subfloat[][]{
    \includegraphics[page=3]{fig1.pdf}
    \label{fig1c}
  }
  \subfloat[][]{
    \includegraphics[page=4]{fig1.pdf}
    \label{fig1d}
  }
  \caption{
    Illustration of our projection method to separate a matrix $P\in\cut{d}{}$ from the correlation polytope $\cut{1}{}$.
    (a) In all cases that we consider in this work, the direction $P$ already defines, together with the corresponding value of $\sdp{1}{P}$, a separating hyperplane.
    (b) When running our FW algorithm to minimise the distance between $P$ and $\cut{1}{}$, the orthogonal projection obtained does not, in general, lie on a facet.
    The resulting separating hyperplane $M$ is better than $P$ from (a), but still not optimal.
    Note that this situation can also happen when starting from a point $vP$ on the line $[0,P]$ that is not close enough to $\cut{1}{}$ (not depicted here).
    (c) Even when $v$ is chosen small enough, if our FW algorithm has not fully converged to the orthogonal projection of $vP$, the separating hyperplane $M$ slightly deviates from a facet.
    (d) For a small enough $v$ and when our FW algorithm has converged, the resulting separating hyperplane is a facet, denoted $A$, that is, the tightest separation leading to the best value of $\innp{A}{P}/\sdp{1}{A}$.
  }
  \label{fig1}
\end{figure}

\begin{algorithm}[H]
  \caption{Iterative procedure to derive a facet via the Blended Pairwise Conditional Gradient (BPCG) algorithm}
  \label{alg:facet}
  \begin{algorithmic}[1]
    \State Input: $P\in\cut{d}{m_1,m_2}$
    \State $v\gets1$ \Comment{$P$ assumed outside of $\cut{1}{m_1,m_2}$, see \cref{fig1a}}
    \State $\mathcal{S}\gets\varnothing$ \Comment{empty active set to start with}
    \While{$\codim\mathcal{S}>1$} \Comment{dimension of the subspace generated by the active set}
      \State $X,\mathcal{S}\gets\mathrm{BPCG}(vP)$ \Comment{$X$ is the projection of $vP$ onto $\cut{1}{m_1,m_2}$ obtained with Algorithm~2 from~\cite{DIB+23}}
      \State $M\gets vP-X$ \Comment{direction of the separating hyperplane given by the final gradient}
      \State $v\gets\sdp{1}{M}/\innp{M}{P}$ \Comment{new point on the line $[0,P]$, strictly closer to $\cut{1}{m_1,m_2}$ than the previous one}
    \EndWhile \Comment{$\codim\mathcal{S}=1$, the ambient dimension may be smaller than $m_1m_2$ by symmetry}
    \State Output: $A\perp\mathcal{S}$ \Comment{normal vector defining the facet $A$ of the symmetrisation of $\cut{1}{m_1,m_2}$~\cite{BBB+12,DVP24}}
  \end{algorithmic}
\end{algorithm}

\subsection{Heuristic results with highly symmetric line packings}
\label{sec:sym}

Starting from a good point $P\in\cut{d}{}$ in the procedure described above is key to obtaining good bounds on $\KG{d}{}$.
Following~\cite{DBV17}, we observed that spanning the $d$-dimensional projective unit sphere in the most uniform way seems to be favourable, which served as our guiding principle to come up with structures in higher dimensions.
In this \cref{sec:sym}, we use the same distribution of points on both sides in \cref{eqn:cut}, that is, $m_1=m_2=m$ and $a_1=b_1,\ldots a_m=b_m$; the matrix $P$ is therefore a Gram matrix.

The intuitive but vague notion of uniform spreading on the sphere can be formally seen as the problem of finding good line packings~\cite{CHS96,FJM18}.
In \cref{tab:sym}, for various symmetric $d$-dimensional line packings, we give the values of $\sdp{d}{P}/\sdp{1}{P}$ and $\sdp{d}{A}/\sdp{1}{A}$.
Since the Gram matrix $P$ is positive semidefinite, the ratio $\sdp{d}{P}/\sdp{1}{P}$ is upper bounded by $K_G^\succeq(d)$, the positive semidefinite Grothendieck constant of order $d$~\cite{BOV10}, whose value is known to be $\gamma(d)\pi/2$~\cite[Theorem~4.1.3]{Bri11}, where $\gamma(d)$ is defined in \cref{eqn:gamma} below.
Note that, for $d\geq9$, this bound is the previously best known lower bound on $\KG{d}{}$ which is naturally greater than $K_G^\succeq(d)$.

Among the structures that we present in \cref{tab:sym}, the $d$-dimensional configurations reaching the best known kissing number in dimension $d$ are of particular interest.
Recall that it corresponds to the greatest number of non-overlapping unit spheres that can be placed such that they each touch (``kiss'') a common unit sphere.
In dimension $d=3$, different realisations of the kissing number give rise to different values of $\sdp{3}{A}/\sdp{1}{A}$.
As the icosahedron is an optimal line packing, this indicates that giving good bounds on $\KG{d}{}$ does not, in general, boil down to finding good line packings.
Note that the unicity of the configuration in dimension $d=4$ was recently proven~\cite{LLM24}.

Upon inspection of \cref{tab:sym} it is apparent that, in the small dimensions that we consider in this work ($d\leq8$), root systems play an important role in obtaining good kissing configurations.
These structures are inherently symmetric and play an important role in the theory of Lie groups and Lie algebras.
We refer to~\cite{Bak02} for definitions and elementary properties.
The family $D_d$ has been studied in~\cite{FR94}; there, starting from the Gram matrix $P$ of size $d(d-1)$, the exact value of $\sdp{1}{P}$ is computed by invoking symmetry arguments.
Moreover, for $d\in\{3, 4, 5\}$, the optimal diagonal modification of $\lambda=\frac23$ is derived for $D_d$; in other words, \cite{FR94} shows that, among all matrices with shape $P-\lambda\id$, the matrix $P-\frac23\id$ is giving the best lower bound on $\KG{d}{}$.
Here we go a bit further by extending this result to $d\in\{6, 7, 8\}$, a fact that was suspected in~\cite{FR94}, but we also observe that the corresponding matrices $A=P-\frac23\id$ are actually facets of the symmetrised polytope $D_d(\cut{1}{d(d-1)})$.

In \cref{tab:sym}, we also give optimal diagonal modifications for other configurations.
For the root system $E_8$, the value $45/31$ reached without diagonal modification is already provided in~\cite[page~50]{FR94}, where it is attributed to Reeds and Sloane.
It indeed follows from \cite[Lemma~2]{FR94} and from the transitivity of the Weyl group, see the acknowledgements in \cref{sec:acks}.
However, these arguments do not directly apply to solve the diagonally modified case.

Note that in all cases involving irrational numbers (except the icosahedron), the facet $A$ cannot be obtained via a diagonal modification of $P$.
This is because the off-diagonal elements of $P$ feature rational and irrational numbers.
The exception of the icosahedron is due to its extra property of being an equiangular tight frame (ETF), all off-diagonal elements being $\pm1/\sqrt5$, which can then lead to a facet by taking an irrational diagonal modification.

\begin{table}[ht!]
  \centering
  \begin{tabular}{|c|c|cl|cl|c|c|c|c|}
    \hline
    ~$d$~                & $m$                 & \multicolumn{2}{c|}{$\frac{\sdp{d}{P}}{\sdp{1}{P}}$}                               & \multicolumn{2}{c|}{$\frac{\sdp{d}{A}}{\sdp{1}{A}}\geq$}                              & Comments              & ~Kissing~    & $G$                     & $\lambda$                \\ \hline
    2                    & 3                   & $\frac98$                               & $=1.125$                                 & $\frac54$                                 & $=1.25\hphantom{00^*}$~\cite{DVP24}       & Hexagon               & $\checkmark$ & ~$A_2$~                 & $\frac23$                \\ \hline
                         &                     & $\frac{9-3\sqrt5}{2}$                   & $\approx1.1459\hphantom{^*}$~\cite{PV22} & $\frac{1+3\sqrt5}{6}$                     & $\approx1.2847$                           & Icosahedron           & $\checkmark$ & $H_3$                   & ~$\frac{15-\sqrt5}{15}$~ \\ \cline{3-10}
                         & \multirow{-2}{*}{6} & $\frac65$                               & $=1.2\hphantom{000^*}$~\cite{FR94}       & $\frac43$                                 & $\approx1.3333\hphantom{^*}$~\cite{FR94}  & Cuboctahedron         & $\checkmark$ & $A_3$                   & $\frac23$                \\ \cline{2-10}
                         & 10                  & $\frac{25(7-3\sqrt5)}{6}$               & $\approx1.2158\hphantom{^*}$~\cite{PV22} & $\frac{7+3\sqrt5}{10}$                    & $\approx1.3708$                           & Dodecahedron          &              &                         &                          \\ \cline{2-8}\cline{10-10}
    \multirow{-4}{*}{3}  & 15                  & $\frac{75(31-12\sqrt5)}{241}$           & $\approx1.2968\hphantom{^*}$             & $\frac{5\sqrt5}{8}$                       & $\approx1.3975$                           & ~Icosidodecahedron~   &              & \multirow{-2}{*}{$H_3$} &                          \\ \hline
                         & 12                  & $\frac97$                               & $\approx1.2857\hphantom{^*}$~\cite{FR94} & $\frac75$                                 & $=1.4\hphantom{000^*}$~\cite{FR94}        & 24-cell               & $\checkmark$ & $D_4$                   & $\frac23$                \\ \cline{2-10}
                         & 60                  & $\frac{225(83-36\sqrt5)}{409}$          & $\approx1.3762\hphantom{^*}$~\cite{PV22} & $\frac{100(16+15\sqrt5)}{3361}$           & $\approx1.4740$                           & 600-cell              &              &                         &                          \\ \cline{2-8}\cline{10-10}
    \multirow{-3}{*}{4}  & ~300~               & $\frac{225(1635-731\sqrt5)}{71}$        & $\approx1.3763\hphantom{^*}$~\cite{PV22} & \cc$\frac{3(\alpha+\beta\sqrt5)}{\gamma}$ & \cc$\approx1.4996^*$                      & 120-cell              &              & \multirow{-2}{*}{$H_4$} &                          \\ \hline
    5                    & 20                  & $\frac43$                               & $\approx1.3333\hphantom{^*}$~\cite{FR94} & $\frac{10}{7}$                            & $\approx1.4286\hphantom{^*}$~\cite{FR94}  &                       & $\checkmark$ & $D_5$                   & $\frac23$                \\ \hline
                         & 30                  & $\frac{15}{11}$                         & $\approx1.3636\hphantom{^*}$~\cite{FR94} & $\frac{13}{9}$                            & $\approx1.4444$                           &                       &              & $D_6$                   & $\frac23$                \\ \cline{2-10}
    \multirow{-2}{*}{6}  & 36                  & $\frac{18}{13}$                         & $\approx1.3846$                          & $\frac{16}{11}$                           & $\approx1.4545$                           &                       & $\checkmark$ & $E_6$                   & $\frac23$                \\ \hline
                         & 28                  & $\frac{28}{25}$                         & $=1.12$                                  & $\frac{133}{109}$                         & $\approx1.2202$                           & ETF~\cite{LS66}       &              &                         & $\frac{39}{24}$          \\ \cline{2-10}
                         & 42                  & $\frac{18}{13}$                         & $\approx1.3846\hphantom{^*}$~\cite{FR94} & $\frac{16}{11}$                           & $\approx1.4545$                           &                       &              & $D_7$                   & $\frac23$                \\ \cline{2-10}
                         & 63                  & $\frac{27}{19}$                         & $\approx1.4211$                          & \cc$\frac{2961}{1991}$                    & \cc$\approx1.4872$                        &                       & $\checkmark$ & $E_7$                   & $\frac76$                \\ \cline{2-10}
    \multirow{-4}{*}{7}  & 91                  & ~$\frac{3549(1433-604\sqrt3)}{959041}$~ & $\approx1.4315^*$                        & \cc~$\frac{24631+18216\sqrt3}{37463}$~    & \cc$\approx1.4997^*$                      & E7 + ETF~\cite{FJM18} &              &                         &                          \\ \hline
                         & 56                  & $\frac75$                               & $=1.4\hphantom{000^*}$~\cite{FR94}       & $\frac{19}{13}$                           & $\approx1.4615$                           &                       &              & $D_8$                   & $\frac23$                \\ \cline{2-10}
    \multirow{-2}{*}{8}  & 120                 & $\frac{45}{31}$                         & $\approx1.4516\hphantom{^*}$~\cite{FR94} & \cc$\frac{165}{109}$                      & \cc$\approx1.5138^*$                      &                       & $\checkmark$ & $E_8$                   & $\frac{13}{6}$           \\ \hline
  \end{tabular}
  \caption{
    Lower bounds on $\KG{d}{}$ obtained via root systems and other remarkable configurations with symmetry group $G$.
    Some structures achieve the best known kissing number in their dimension, which we indicate with $\checkmark$.
    The matrix $A$ is the facet separating the Gram matrix $P$ of the configuration from the symmetrised correlation polytope $G(\cut{1}{})$.
    When $A=P-\lambda\id$, we give the corresponding value of this \emph{diagonal modification}~\cite{FR94} in the last column.
    Shaded cells indicate potential improvements on the state of the art, but in most of these cases, the asterisk $^*$ indicates that our solution of $\sdp{1}{A}$ is heuristic, hence only putative, see \cref{tab:heuristic}.
    Only the bound on $\KG{7}{}$ obtained with the $E_7$ root system is both exact and good enough to beat the literature; below, by monotonicity of $\KG{d}{}$, we further improve this bound, see \cref{tab:exact}.
    For the 120-cell, we define $\alpha=2566372165103191$, $\beta=1178280120531798$, and $\gamma=10405220765436757$.
    Note that we are able to exactly obtain $\sdp{d}{P}$ as all these configurations satisfy the \emph{platonic} property from~\cite{PV22}.
    Moreover, we could numerically observe the tightness of our lower bound on $\sdp{d}{A}$ by computing the first level of the Lasserre hierarchy~\cite{Las01}, see \cref{sec:KGd2_low} below.
  }
  \label{tab:sym}
\end{table}

The 600-cell and the 120-cell have already been studied in~\cite{PV22}, where symmetry arguments are exploited to compute the value of $\sdp{1}{P}$, even for the $300\times300$ matrix arising from the 120-cell.
Interestingly, although the difference in value for $\sdp{4}{P}/\sdp{1}{P}$ is almost negligible, the facets that we obtain here ``activate'' the advantage of the 120-cell.
Also, we emphasise the advantage of our geometric approach for these structures: in~\cite{PV22}, the optimal diagonal modification of $23/9$ was indeed derived for the 600-cell, giving rise to the value of $35(-37+27\sqrt5)/569\approx1.4378$, which is significantly smaller than our result of about $1.4740$, see \cref{tab:sym}.
This is because we are not restricted to diagonal modifications in our algorithm, which gives a significant superiority for these configurations.

For some instances of $\sdp{1}{P}$ and $\sdp{1}{A}$, the size of the matrix is too large for numerical solvers to handle it.
The values presented with an asterisk in \cref{tab:sym} and summarised in \cref{tab:heuristic} are then obtained heuristically and cannot be considered final results until these values are confirmed to be optimal.
For this, the high symmetry of these instances, detrimental for our branch-and-bound algorithm and not exploited in available QUBO/MaxCut solvers, could play a crucial role.
We leave this open for further research.

\subsection{Exact results in asymmetric scenarios}
\label{sec:asy}

Now, we discuss how to address \cref{pb:qubo}, namely, the computation of
\begin{equation}
  \sdp{1}{M}=\max_{\hb{a_x=\pm1}{b_y=\pm1}}\sum_{x=1}^{m_1}\sum_{y=1}^{m_2} M_{xy}{a_x}{b_y}.
  \label{eqn:qubo}
\end{equation}
In~\cite{DIB+23} this problem is reformulated into a Quadratic Unconstrained Binary Optimisation (QUBO) instance, which is then given to the solver QuBowl from~\cite{RKS23}.
The instance solved there involves $m_1=m_2=97$ and is bigger than the ones from~\cite{DBV17} (of maximal size $m_1=m_2=92$).
However, it is worth noticing that the solution of the exact instance solved in~\cite{DIB+23} stands out in the landscape of all binary variables.
More precisely, the exact solution is likely to be unique and its value is far above the other feasible points, so that the solver could efficiently exclude vast portions of the search space.
This empirical observation is supported by two facts: instances from later stages of convergence of the Frank-Wolfe method cannot be solved by QuBowl (even within ten times as much time), and the symmetric instance $A$ of size $63\times63$ obtained from the $E_7$ root system (see \cref{tab:sym}) cannot be solved by QuBowl either, although it is way smaller.

Given these difficulties, we consider other formulations exploiting the specificities of the problem.
In particular, the branch-and-bound algorithm from~\cite{DBV17} works by breaking the symmetry between the binary variables $a_x$ and $b_y$ in \cref{eqn:qubo}, fixing the latter to obtain:
\begin{equation}
  \sdp{1}{M}=\max_{a_x=\pm1}\sum_{y=1}^{m_2}\left|\sum_{x=1}^{m_1} M_{xy}{a_x}\right|,
  \label{eqn:bnb}
\end{equation}
which has half as many variables.
Importantly, the parameter $m_2$ plays a less critical role in the complexity of the resulting algorithm, which is indeed still exponential in $m_1$, but now linear in $m_2$.
This suggests using rectangular matrices $M$ in our procedure, with a large $m_2$ to increase the achievable lower bound on $\KG{d}{}$ and a relatively small $m_1$ to maintain the possibility to solve \cref{pb:qubo}, that is, to compute $\sdp{1}{M}$.

Such rectangular matrices can be obtained by following the procedure described in \cref{sec:fw} starting from rectangular matrices $P\in\cut{d}{m_1,m_2}$.
Good starting points are still obtained by using well-spread distributions on the sphere, but these distributions $a_1,\ldots,a_{m_1}$ and $b_1,\ldots,b_{m_2}$ are now different, so that the matrix $P$ with entries $\innp{a_x}{b_y}$ is not a Gram matrix any more.
We give our best lower bounds on $\KG{d}{}$ in \cref{tab:exact} and present relevant details below.
All matrices can be found in the supplementary files~\footnote{Supplementary files can be found on Zenodo: \url{https://doi.org/10.5281/zenodo.13693164}.}.

In dimension $d=3$, despite our efforts, we could not find a rectangular instance beating the $97\times97$ from~\cite{DIB+23}.
However, the quantum point provided therein, namely, the polyhedron on the Bloch sphere consisting of $97$ pairs of antipodal points, is not the best quantum strategy to violate the inequality.
By using the Lasserre hierarchy at its first level, we can indeed obtain a slightly better violation, hence a tighter lower bound on $\KG{3}{}$.
Summarising, with the matrix $M$ from~\cite{DIB+23} that had been constructed starting from a Gram matrix $P\in\cut{3}{97}$, we have:
\begin{itemize}
  \item $\sdp{3}{M}\geq\innp{M}{P}\approx2.0000\times10^{22}$,
  \item $\sdp{3}{M}\geq\innp{M}{P'}\approx2.0001\times10^{22}$ where $P'\in\cut{3}{97}$ is obtained at the first level of the Lasserre hierarchy (rational expression in the supplementary file~\cite{Note1}),
  \item $\sdp{1}{M}=13921227005628453160441$ (solved with QuBowl~\cite{RKS23}).
\end{itemize}

In dimension $d=4$, we start with the 600-cell on one side and use the compound formed by the same 600-cell together with its dual (the 120-cell) on the other side.
This creates a matrix $P\in\cut{4}{60,360}$ we can run our separation procedure on.
Similarly to~\cite{DVP24}, we exploit the symmetry of $P$ throughout the algorithm to reduce the dimension of the space and accelerate the convergence.
The resulting matrix $A$ already has integer coefficients and satisfies:
\begin{itemize}
  \item $\sdp{4}{A}\geq\innp{A}{P}=30\left(227668+322725\sqrt2+170064\sqrt5+182375\sqrt{10}\right)$,
  \item the first level of the Lasserre hierarchy numerically confirms this value as optimal,
  \item $\sdp{1}{A}=33135128$ (solved with the branch-and-bound algorithm from~\cite{DBV17}).
\end{itemize}

In dimension $d=5$, we make use of structures in \href{http://neilsloane.com/grass/}{Sloane's database}~\cite{CHS96}.
On one side, we directly use their five-dimensional structure with 65 lines.
On the other side, we take their five-dimensional structure with 37 lines and augment it by adding the centres of all edges, properly renormalised to lie on the sphere, which creates a five-dimensional structure with 385 lines.
The resulting matrix $P\in\cut{5}{65,385}$ goes through our separation algorithm until a satisfactory precision is attained.
After rounding we obtain a matrix $M$ satisfying:
\begin{itemize}
  \item $\sdp{5}{M}\geq\innp{M}{P}\approx2.1061\times10^8$,
  \item $\sdp{5}{M}\geq\innp{M}{P'}\approx2.1068\times10^8$ where $P'\in\cut{5}{65,385}$ is obtained at the first level of the Lasserre hierarchy (rational expression in the supplementary files~\cite{Note1}),
  \item $\sdp{1}{M}=141074623$ (solved with the branch-and-bound algorithm from~\cite{DBV17}).
\end{itemize}

\begin{table}[ht!]
  \centering
  \subfloat[][Heuristic]{
    \begin{tabular}{|c|c|c|c|c|}
      \hline
      ~$d$~ & SoA                         & ~$m_1$~ & ~$m_2$~ & ~$\KG{d}{}\geq$~ \\ \hline
      4     &                             & 300     & 300     & 1.49956          \\ \cline{1-1}\cline{3-5}
      7     &                             & 91      & 91      & 1.49967          \\ \cline{1-1}\cline{3-5}
      8     & \multirow{-3}{*}{~1.48217~} & 120     & 120     & 1.51376          \\ \hline
    \end{tabular}
    \label{tab:heuristic}
  }
  \hspace{1cm}
  \subfloat[][Exact]{
    \begin{tabular}{|c|c|c|c|c|}
      \hline
      ~$d$~ & SoA                       & ~$m_1$~ & ~$m_2$~ & ~$\KG{d}{}\geq$~ \\ \hline
      3     & ~1.43665~                 & 97      & 97      & 1.43670          \\ \hline
      4     &                           & 60      & 360     & 1.48579          \\ \cline{1-1}\cline{3-5}
      5     & \multirow{-2}{*}{1.48217} & 65      & 385     & 1.49339          \\ \hline
    \end{tabular}
    \label{tab:exact}
  }
  \caption{
    Better lower bounds on $\KG{d}{}$ compared to the state of the art (SoA) from~\cite{DIB+23,DBV17}.
    They are obtained by choosing two line structures in dimension $d$ defining a scalar product matrix $P\in\cut{d}{m_1,m_2}$, and then by running our Frank-Wolfe method to obtain a good separating hyperplane $M$ with respect to the correlation polytope $\cut{1}{m_1,m_2}$ (sometimes even a facet), see \cref{sec:fw}.
    Eventually, we compute the first level of the Lasserre hierarchy to either numerically confirm the optimality of $P$ in \cref{eqn:SDPdM_cut} or build a slightly better matrix $P'\in\cut{d}{m_1,m_2}$.
    In \cref{tab:exact}, we could solve $\sdp{1}{M}$ and $\KG{d}{}\geq\sdp{d}{M}/\sdp{1}{M}$ immediately follows from \cref{eqn:KGdn}.
    In \cref{tab:heuristic}, we could not solve $\sdp{1}{M}$ so that the bounds shown are heuristic; see \cref{tab:sym} for analytical values.
    Note that the correct value of the SoA for $\KG{4}{}$ in~\cite{DBV17} is taken from Section~III~B therein and not from the erroneous abstract.
    By monotonicity, our bound on $\KG{5}{}$ propagates to all $\KG{d}{}$ for $d\geq5$; it therefore beats the best bounds on Grothendieck constants of order up to $d=9$, where $\KG{9}{}\geq1.48608$ comes from~\cite{BBT11}.
    Moreover, it also beats the value $\KG{7}{}\geq1.48719$ that we had found with the $E_7$ root system in \cref{tab:sym}.
  }
  \label{tab:asy}
\end{table}

\section{Bounds on \texorpdfstring{$\KG{d}{2}$}{KG(d->2)}}
\label{sec:KGd2}

We now turn to the generalised Grothendieck constants $\KG{d}{2}$, which have not received a lot of attention so far.
Motivated by the quantum interpretation of these constants (see \cref{sec:quantum}), we extend the techniques presented above and in previous works to refine the bounds known on their values.

\subsection{Bounds arising from the literature}

There is an explicit lower bound on $\KG{d}{n}$ for any $d>n$ that uses a particular one-parameter family of linear functionals of infinite size.
This bound, initially proven in~\cite{VP09,BBT11} (see also~\cite[Theorem~3.3.1]{Bri11} for the proof), reads
\begin{equation}
  \KG{d}{n}\geq\frac{\gamma(d)}{\gamma(n)},\quad\text{with}\quad\gamma(d)=\frac{2}{d}\left(\frac{\Gamma\left(\frac{d+1}{2}\right)}{\Gamma\left(\frac{d}{2}\right)}\right)^2,
  \label{eqn:gamma}
\end{equation}
where $\Gamma(z)=\int_0^\infty t^{z-1}\be^{-t}\bd t$ is the gamma function.
When $n=2$, $\gamma(n)=\pi/4$ and \cref{eqn:gamma} can be rewritten:
\begin{equation}
  \KG{d}{2}\geq\frac{d\binom{d-1}{k}^2}{2^{2d-3}}\quad\text{when}\quad d=2k\qquad\text{and}\qquad\KG{d}{2}\geq\frac{2^{2d+1}}{d\binom{d-1}{k}^2\pi^2}\quad\text{when}\quad d=2k+1.
  \label{eqn:SoA_n2}
\end{equation}

The bounds in \cref{eqn:SoA_n2} are better than indirect bounds derived from natural combinations of bounds on the Grothendieck constants $\KG{d}{}$, even when considering the improved bounds proven in this article.
Starting with the Grothendieck inequality~\eqref{eqn:Grothendieck}, dividing both sides by $\sdp{1}{M}$ and lower and upper bounding the left and right sides, we indeed obtain the following relation
\begin{equation}
  \KG{d}{}\leq \KG{d}{n}\KG{n}{}.
  \label{eqn:multip}
\end{equation}
Then, inserting $\KG{2}{}=\sqrt2$ from~\cite{Kri79} and using our lower bounds presented in \cref{tab:exact} yields $\KG{3}{2}\geq1.0159$, $\KG{4}{2}\geq1.0506$, and $\KG{5}{2}\geq1.0560$, which does not improve on the values from \cref{eqn:SoA_n2}.

As can be seen in \cref{tab:literature_n2}, \cref{eqn:SoA_n2} also outperforms all known bounds obtained with finite matrices.
Therefore, to the best of our knowledge, \cref{eqn:SoA_n2} represents the state of the art.

\begin{table}[ht!]
  \centering
  \begin{tabular}{|c|cl|c|c|c|c|c|c|}
    \hline
    ~$d$~                & \multicolumn{2}{c|}{SoA}                                                    & ~$m_1$~ & ~$m_2$~ & ~$\sdp{d}{M}\geq$~ & ~$\sdp{2}{M}\leq$~ & ~$\KG{d}{2}\geq$~ & Reference     \\ \hline
                         &                                        &                                    & 3       & 6       & $6\sqrt2$          & $3(1+\sqrt3)$      & 1.0353            & \cite{BSC+18} \\ \cline{4-9}
                         &                                        &                                    & 4       & 8       & 15.4548            & 14.8098            & 1.0436            & \cite{VP08}   \\ \cline{4-9}
     \multirow{-3}{*}{3} & \multirow{-3}{*}{~$\frac{32}{3\pi^2}$} & \multirow{-3}{*}{$\approx1.0808$~} & 3       & 4       & $4\sqrt3$          & $2(1+\sqrt5)$      & 1.0705            & \cite{Gis09}  \\ \hline
     4                   & $\frac98$                              & $=1.125\hphantom{0}$               & 4       & 8       & 16                 & 14.8098            & 1.0804            & \cite{VP08}   \\ \hline
  \end{tabular}
  \caption{
    Lower bounds on $\KG{d}{2}$ found in the literature when considering finite matrices and compared with the state of the art proven by using ``infinite matrices'' in~\cite{VP09,BBT11} and given in \cref{eqn:SoA_n2}.
  }
  \label{tab:literature_n2}
\end{table}

As far as we authors know, there are no nontrivial upper bounds on $\KG{d}{n}$ readily available in the literature.
The higher complexity of finding such bounds is partly explained by the natural direction imposed by the definition of $\KG{d}{n}$ as a supremum, see \cref{eqn:KGdn}.
In \cref{sec:up_n2} below, we give a general algorithmic method to overcome this difficulty.

\subsection{Numerical lower bounds on \texorpdfstring{$\KG{d}{2}$}{KG(d->2)}}
\label{sec:KGd2_low}

Similarly to \cref{sec:KGd_low}, providing a matrix $M$ together with an upper bound on $\sdp{2}{M}$ and a lower bound on $\sdp{d}{M}$ suffices to obtain a lower bound on $\KG{d}{2}$.
However, solving the optimisation problem $\sdp{2}{M}$ is even harder than $\sdp{1}{M}$: it is also nonlinear and nonconvex, but with continuous variables instead of binary ones.
In this section we explain how we obtain numerical upper bounds on $\sdp{2}{M}$ and summarise our results in \cref{tab:sym_n2}.

We first rewrite \cref{eqn:SDPdM} by introducing the (real) coordinates of the $d$-dimensional vectors $a_x$ and $b_y$:
\begin{equation}
  \sdp{d}{M}=\max\left\{\sum_{x=1}^{m_1}\sum_{y=1}^{m_2} M_{xy}\innp{\begin{pmatrix}a_{x,1}\\\vdots\\a_{x,d}\end{pmatrix}}{\begin{pmatrix}b_{y,1}\\\vdots\\b_{y,d}\end{pmatrix}}~\Bigg|~\forall x\in[m_1],~\sum_{i=1}^da_{x,i}^2=1,~\forall y\in[m_2],~\sum_{i=1}^db_{y,i}^2=1\right\}.
\end{equation}
With this reformulation, we see that obtaining upper bounds on $\sdp{d}{M}$ amounts to solving the Lasserre hierarchy of this polynomial system with $d(m_1+m_2)$ variables and $m_1+m_2$ constraints at a certain level~\cite{Las01}.
At the first level, this method is used in \cref{tab:sym} to numerically confirm the value of $\sdp{d}{A}$, as the upper bound computed matches, up to numerical accuracy, our analytical lower bound.
In the following, we use it at the second level and for $n=2$ to numerically obtain upper bounds on $\sdp{2}{M}$.

Similarly to \cref{tab:sym}, the different values that we reach are given in \cref{tab:sym_n2}.
The solution to the Lasserre hierarchy was obtained with the implementation in~\cite{WML21a}.
Importantly, these results suffer from numerical imprecision.
This is all the more relevant here because of the huge size of most instances, making the use of a first-order solver, COSMO~\cite{GCG21} in our case, almost mandatory, which is detrimental to the precision of the optimum returned.
This explains the deviations observed in \cref{tab:sym_n2}: the fourth digit of the bound computed via the Lasserre hierarchy is not significant.

Interestingly, the icosahedron now provides us with a better bound than the cuboctahedron, whereas this was the opposite in \cref{tab:sym}.
The dodecahedron also gives no advantage over the icosahedron, up to numerical precision at least.
Similarly, the 600-cell and the 120-cell give very close bounds, which is in strong contrast with \cref{tab:sym}.

Note that the bounds presented in \cref{tab:sym_n2} allow us to prove that \cref{eqn:multip} is strict for $n=2$ and $d=3$, a fact that could not be verified with \cref{eqn:SoA_n2}.
In this case, we indeed have that $\KG{3}{}\leq1.455$ from~\cite{DIB+23} and that $\KG{2}{}\KG{3}{2}\geq\sqrt2\times1.103\approx1.560$ from~\cite{Kri79} and with the bound given in \cref{tab:sym_n2}.

\begin{table}[ht!]
  \centering
  \begin{tabular}{|c|cl|c|c|c|c|c|c|c|}
    \hline
    ~$d$~               & \multicolumn{2}{c|}{SoA}                                                            & $m$                 & ~Heuristic~ & ~Lasserre~ & Comments              & ~Kissing~    & $G$                     & $\lambda$ \\ \hline
                        &                                            &                                        &                     & 1.0560      & 1.0560     & Cuboctahedron         & $\checkmark$ & ~$A_3$~                 & ~0.9209~  \\ \cline{5-10}
                        &                                            &                                        & \multirow{-2}{*}{6} & \cc1.0983   & \cc1.0983  & Icosahedron           & $\checkmark$ &                         & 0.8252    \\ \cline{4-8}\cline{10-10}
                        &                                            &                                        & 10                  & \cc1.0983   & \cc1.0983  & Dodecahedron          &              &                         &           \\ \cline{4-8}\cline{10-10}
    \multirow{-4}{*}{3} & \multirow{-4}{*}{$\frac{32}{3\pi^2}$}      & \multirow{-4}{*}{$\approx1.0808$}      & 15                  & \cc1.1027   & \cc1.1028  & ~Icosidodecahedron~   &              & \multirow{-3}{*}{$H_3$} &           \\ \hline
                        &                                            &                                        & 12                  & 1.1181      & 1.1182     & 24-cell               & $\checkmark$ & $D_4$                   & 1.0981    \\ \cline{4-10}
                        &                                            &                                        & 60                  & \cc1.1704   &            & 600-cell              &              &                         &           \\ \cline{4-8}\cline{10-10}
    \multirow{-3}{*}{4} & \multirow{-3}{*}{$\frac98$}                & \multirow{-3}{*}{$=1.125\hphantom{0}$} & ~300~               & \cc1.1741   &            & 120-cell              &              & \multirow{-2}{*}{$H_4$} &           \\ \hline
    5                   & $\frac{512}{45\pi^2}$                      & $\approx1.1528$                        & 20                  & \cc1.1540   & 1.1469     &                       & $\checkmark$ & $D_5$                   & 1.2654    \\ \hline
                        &                                            &                                        & 30                  & 1.1718      & 1.1631     &                       &              & $D_6$                   & 1.4402    \\ \cline{4-10}
    \multirow{-2}{*}{6} & \multirow{-2}{*}{$\frac{75}{64}$}          & \multirow{-2}{*}{$\approx1.1719$}      & 36                  & \cc1.1846   &            &                       & $\checkmark$ & $E_6$                   & 1.4677    \\ \hline
                        &                                            &                                        & 28                  & 1.1435      & 1.1266     & ETF~\cite{LS66}       &              &                         & 1.5738    \\ \cline{4-10}
                        &                                            &                                        & 42                  & 1.1830      &            &                       &              & $D_7$                   & 1.5481    \\ \cline{4-10}
                        &                                            &                                        & 63                  & \cc1.2060   &            &                       & $\checkmark$ & $E_7$                   & 1.8025    \\ \cline{4-10}
    \multirow{-4}{*}{7} & \multirow{-4}{*}{~$\frac{2048}{175\pi^2}$} & \multirow{-4}{*}{$\approx1.1857$~}     & 91                  & \cc1.2139   &            & E7 + ETF~\cite{FJM18} &              &                         &           \\ \hline
                        &                                            &                                        & 56                  & 1.1913      &            &                       &              & $D_8$                   & 1.6380    \\ \cline{4-10}
    \multirow{-2}{*}{8} & \multirow{-2}{*}{$\frac{1225}{1024}$}      & \multirow{-2}{*}{$\approx1.1963$}      & 120                 & \cc1.2251   &            &                       & $\checkmark$ & $E_8$                   & 3.0176    \\ \hline
  \end{tabular}
  \caption{
    Lower bounds on $\KG{d}{2}$ obtained via root systems and other remarkable configurations with symmetry group $G$.
    ``Heuristic'' and ``Lasserre'' refer to the method we use to derive a lower bound on $\sdp{d}{M}/\sdp{2}{M}$.
    The matrix $M$ is computed by separating the Gram matrix $P$ of the configuration from the symmetrisation of $\cut{2}{}$.
    When $M=P-\lambda\id$, we give the corresponding value of this \emph{diagonal modification}~\cite{FR94} in the last column.
    Shaded cells indicate (numerical) improvements on the state of the art from \cref{eqn:SoA_n2}, shown on the left for comparison.
    Note that the Lasserre hierarchy~\cite{Las01}, up to the numerical inaccuracy emphasised in the main text and giving rise to small inconsistencies between the ``Heuristic'' and ``Lasserre'' columns, provides us with a certified bound.
    Here we solve it at its second level, except for $m=6$, in which case we reach the third level.
  }
  \label{tab:sym_n2}
\end{table}

\subsection{Exact upper bound on \texorpdfstring{$\KG{3}{2}$}{KG(3->2)}}
\label{sec:up_n2}

Here we reformulate the procedure first introduced in~\cite{CGRS16,HQV+16} and later used to obtain better upper bounds on $\KG{3}{}$ in~\cite{HQV+17,DIB+23} in a way that makes the generalisation to other Grothendieck constants more transparent.

\begin{proposition}
  Let $P\in\cut{d}{p_1,p_2}$ with underlying vectors $a_1^P,\ldots,a_{p_1}^P,b_1^P,\ldots,b_{p_2}^P\in S^{d-1}$ such that
  \begin{itemize}
    \item there exist $\eta_1,\eta_2\in(0,1)$ for which $\eta_1S^{d-1}\subset\conv\{a_i^P\}_{i=1}^{p_1}$ and $\eta_2S^{d-1}\subset\conv\{b_j^P\}_{j=1}^{p_2}$,
    \item there exists $\alpha\in(0,1)$ such that $\alpha P\in\cut{n}{p_1,p_2}$.
  \end{itemize}
  Then we have the following bound:
  \begin{equation}
    \KG{d}{n}\leq\frac{1}{\alpha\,\eta_1\eta_2}.
    \label{eqn:upper}
  \end{equation}
  \label{prop:upper}
\end{proposition}

\begin{proof}
  Consider a matrix $M$ of size $m_1\times m_2$.
  Since
  \begin{equation}
    \alpha\,\eta_1\eta_2\,\sdp{d}{M}=\alpha\,\eta_1\eta_2\max_{N\in\cut{d}{m_1,m_2}}\innp{M}{N}=\max_{N\in\cut{d}{m_1,m_2}}\innp{M}{\alpha\,\eta_1\eta_2\,N},
    \label{eqn:proof_upper1}
  \end{equation}
  we focus on $\alpha\,\eta_1\eta_2\,N$ for $N\in\cut{d}{m_1,m_2}$ with underlying vectors $a_1^N,\ldots,a_{m_1}^N,b_1^N,\ldots,b_{m_2}^N\in S^{d-1}$.

  By the first assumption on $\{a_i^P\}_{i=1}^{p_1}$ and $\{b_j^P\}_{j=1}^{p_2}$, we know that we can write, for $x\in[m_1]$ and $y\in[m_2]$,
  \begin{equation}
    \eta_1a_x^N=\sum_{i=1}^{p_1}v_i^{(x)}a_i^P\quad\text{and}\quad\eta_2b_y^N=\sum_{j=1}^{p_2}w_j^{(y)}b_j^P,\quad\text{so that}\quad\alpha\,\eta_1\eta_2\,N_{xy}=\alpha\sum_{i,j}v_i^{(x)}w_j^{(y)}\innp{a_i^P}{b_j^P}
  \end{equation}
  where $v_i^{(x)}\geq0$, $w_j^{(y)}\geq0$, $\sum_iv_i^{(x)}=1$, and $\sum_jw_j^{(y)}=1$.

  Next we use the second assumption, which can be written as follows: there exist positive weights $\lambda_k$ summing up to one and vectors $\ta_1^{(k)},\ldots,\ta_{p_1}^{(k)},\tb_1^{(k)},\ldots,\tb_{p_2}^{(k)}\in S^{n-1}$ indexed by $k$ (in a finite set) such that $\alpha\tinnp{a_i^P}{b_j^P}=\sum_k\lambda_k\tinnp{\ta_i^{(k)}}{\tb_j^{(k)}}$ holds for all $i\in[p_1]$ and $j\in[p_2]$.
  This gives, for $x\in[m_1]$ and $y\in[m_2]$,
  \begin{equation}
    \alpha\,\eta_1\eta_2\,N_{xy}=\sum_{i,j}v_i^{(x)}w_j^{(y)}\sum_k\lambda_k\innp{\ta_i^{(k)}}{\tb_j^{(k)}}=\sum_k\lambda_k\innp{\sum_iv_i^{(x)}\ta_i^{(k)}}{\sum_jw_j^{(y)}\tb_j^{(k)}}=\sum_k\lambda_k\innp{\ta_x^{(k)}}{\tb_y^{(k)}},
    \label{eqn:proof_upper2}
  \end{equation}
  where we have defined the vectors $\ta_x^{(k)}=\sum_iv_i^{(x)}\ta_i^{(k)}$ and $\tb_y^{(k)}=\sum_jw_j^{(y)}\ta_j^{(k)}$ in the $n$-dimensional unit ball $B^n$.

  Combining \cref{eqn:proof_upper1} with \cref{eqn:proof_upper2} and using the definition of $\sdp{d}{M}$ in \cref{eqn:SDPdM}, we obtain
  \begin{equation}
    \alpha\,\eta_1\eta_2\,\sdp{d}{M}\leq\max\left\{\sum_{x=1}^{m_1}\sum_{y=1}^{m_2} M_{xy}\innp{\ta_x}{\tb_y}~\Big|~\forall x\in[m_1],~\ta_x\in B^n,~\forall y\in[m_2],~\tb_y\in B^n\right\}\leq\sdp{n}{M},
  \end{equation}
  so that \cref{eqn:upper} follows from the definition of $\KG{d}{n}$.
\end{proof}

To obtain our bound on $\KG{3}{2}$, we start from the Gram matrix $P$ obtained from a centrosymmetric polyhedron with 812 vertices, that is, 406 lines.
Note that this polyhedron is different from the one used in~\cite{DIB+23} and features a slightly higher shrinking factor through the following construction: starting from an icosahedron, apply four times the transformation adding the normalised centres of each (triangular) facet.
Taking $v_0=0.8962$ and running our Frank-Wolfe algorithm, we obtain a decomposition using 7886 extreme points of $\cut{2}{406}$ and recovering $v_0P$ up to a Euclidean distance of $\varepsilon\approx2.7\times10^{-4}$.
This algorithm is strictly identical to the one in~\cite{DIB+23} except for the LMO, which corresponds to the case described in \cref{sec:fw} when $n=2$.
Following~\cite[Appendix~D]{DIB+23} we can make this decomposition purely analytical by noting that the unit ball for the Euclidean norm is contained in $\cut{1}{}$ and therefore in $\cut{2}{}$.
Note that, contrary to~\cite{DIB+23} where the extremal points of $\cut{1}{}$ were automatically analytical as they only have $\pm1$ elements, an extra step is needed here to make numerical extremal points of $\cut{2}{}$ analytical.
Here we follow a rationalisation procedure very similar to the one described in~\cite[Appendix~A]{DIB+23} and we obtain a value $\alpha=v_0/(1+\varepsilon)$ that can be used in \cref{prop:upper} to derive the bound
\begin{equation}
  \KG{3}{2}\leq1.1233\ldots
  \label{eqn:lowerKG32}
\end{equation}
whose analytical expression is provided in the supplementary files~\cite{Note1}.
This upper bound is indeed entirely rigorous, contrary to the lower bounds presented above in \cref{sec:KGd2_low}, which rely on numerical methods that we did not convert into exact results.

\subsection{Implications for Bell nonlocality}
\label{sec:quantum}

The connection between $\KG{d}{}$ and Bell nonlocality was first established by Tsirelson~\cite{Tsi87} and was later investigated more thoroughly in~\cite{AGT06}.
We refer to the appendix of~\cite{VP08} for a detailed construction of the quantum realisation of the Bell inequalities constructed in our work.
Note that the new lower bound mentioned in \cref{sec:asy} slightly improves on the upper bound on $v_c^\mathrm{Wer}=1/\KG{3}{}$, namely, it reduces it from about $0.69606$ to $0.69604$; see also~\cite{DIB+23} for the definition of this number and its interpretation as a noise robustness.

Works anticipating this connection with $\KG{d}{2}$ go back to 2008~\cite{PV08,VP08}, when Vértesi and Pál investigated the realisation of Bell inequalities in real quantum mechanics and showed that $d$-dimensional complex systems can outperform real ones.
This line of research was further developed by Briët et~al.~\cite{BBT11}.

We refer to~\cite{VP08} for the explicit connection; for this, the examination of \cref{tab:literature_n2} should be of some help as it gives the bridge between their notations and ours.

\section{Discussion}
\label{sec:discussion}

We generalised the method used in~\cite{DIB+23} to obtain lower bounds on $\KG{3}{}$ to Grothendieck constants of higher order $\KG{d}{}$.
We obtained certified lower bounds on $\KG{3}{}$, $\KG{4}{}$, and $\KG{5}{}$ beating previous ones, and setting, by monotonicity, the best known lower bounds on $\KG{d}{}$ for $3\leq d\leq9$.
For these constants of higher orders, we also showed how our Frank-Wolfe algorithm can be used to derive putative facets, but solving the corresponding quadratic binary problems remains open.
The instances given in this article would allow immediate improvement on $\KG{4}{}$, $\KG{7}{}$, and $\KG{8}{}$.
We expect future work to exploit their symmetry to make the computation possible.

Beyond the perhaps anecdotal improvement brought to their respective constants, developing such tools could unlock access to higher-dimensional structures with a size way larger than anything that numerical methods could ever consider solving.
This could in turn give rise to lower bounds on finite-order Grothendieck constants strong enough to compete with the best known lower bound on \emph{the} (infinite-order) Grothendieck constant, namely, the one presented by Davie in 1984~\cite{Dav84} and independently rediscovered by Reeds in 1991~\cite{Ree91}, both works being unpublished.
This bound reads
\begin{equation}
  K_G\geq\sup_{0<\lambda<1}\frac{1-\rho(\lambda)}{\max\{\rho(\lambda),F_{\rho(\lambda)}(\lambda)\}}\approx1.676956674\ldots,
  \quad\text{attained for}\quad
  \lambda\approx0.255730213\ldots
  \label{eqn:infinite}
\end{equation}
where
\begin{equation}
  \rho(\lambda)=\sqrt{\frac{2}{\pi}}\lambda\be^{-\frac{\hphantom{^2\!\!}\lambda^2\!\!}{2}}
  \quad\text{and}\quad
  F_\rho(\lambda)=\frac{2}{\pi}\be^{-\lambda^2}+\rho\left(1-2\sqrt{\frac{2}{\pi}}\int_\lambda^\infty\be^{-\frac{\hphantom{^2\!\!}x^2\!\!}{2}}\bd x\right).
\end{equation}

We also extended the procedure of~\cite{DIB+23} to derive bounds on the generalised constant $\KG{d}{2}$.
These constants appear quite naturally when considering the advantage of quantum mechanics over real quantum mechanics.
In particular, our analytical upper bound on $\KG{3}{2}$ delivers some quantitative insight on this question, while our numerical lower bounds on $\KG{d}{2}$ could prove useful when looking for estimates of the benefit of considering higher-dimensional quantum systems.

We also note that our techniques naturally apply in the complex case, for which similar Grothendieck constants have also been defined.
However, since the complex finite-order Grothendieck constants do not have known interpretations, we refrained from deriving bounds on them, although our code is already capable of dealing with this case.

\section{Code availability}

The code used to obtain the results presented in this article is part of the \texttt{Julia} package \texttt{BellPolytopes.jl} introduced in~\cite{DIB+23} and based on the \texttt{FrankWolfe.jl} package~\cite{BCP22,BDH+24}.
While the algorithm used to derive bounds on $\KG{d}{}$ is directly available in the standard version of this package, the extension to generalised constants $\KG{d}{n}$ can be found on the branch \texttt{mapsto} of this repository.
Installing this branch can be directly done within the \texttt{Julia} package manager by typing \texttt{add BellPolytopes\#mapsto}.

\section{Acknowledgements}
\label{sec:acks}

The authors thank
Péter Diviánszky for fruitful discussions and his help to run the branch-and-bound algorithm from~\cite{DBV17} on asymmetric instances;
James A.~Reeds and Neil J.~A.~Sloane for sharing the details of the derivation of the value $45/31$ for the $E_8$ root system mentioned in~\cite{FR94};
Timotej Hrga, Nathan Krislock, Thi Thai Le, and Daniel Rehfeldt for trying to solve our symmetric instances with various solvers;
Fernando Mário de Oliveira Filho for asking for details about the proof of \cref{prop:upper}, which helped us improve its presentation;
Mathieu Besançon, Daniel Brosch, Wojciech Bruzda, Dardo Goyeneche, Volker Kaibel, David de Laat, Frank Valentin, and Karol Życzkowski for inspiring discussions.
This research was partially supported by the DFG Cluster of Excellence MATH+ (EXC-2046/1, project id 390685689) funded by the Deutsche Forschungsgemeinschaft (DFG).
T.~V.~acknowledges the support of the European Union (QuantERA eDICT, CHIST-ERA MoDIC), the National Research, Development and Innovation Office NKFIH (Grants No.~2019-2.1.7-ERA-NET-2020-00003, No.~2023-1.2.1-ERA\_NET-2023-00009 and No.~K145927) and the `Frontline' Research Excellence Program of the NKFIH (No.~KKP133827).
GPT 6 Astra was used to correct minor grammatical and mathematical errors.

\bibliography{DVP26}

\end{document}